\documentclass[journal,twoside,web]{ieeecolor}
\usepackage{generic}
\usepackage{cite}
\usepackage{amsmath,amssymb,amsfonts}
\usepackage{algorithmic}
\usepackage{graphicx}
\usepackage{textcomp}
\usepackage{bm}
\usepackage{mathrsfs}
\usepackage{graphicx,cite}
\usepackage{amsmath}
\usepackage{amssymb}
\usepackage[dvips]{epsfig}
\usepackage{latexsym}
\usepackage{psfrag}
\usepackage{graphicx,subfigure}
\usepackage{hyperref}
\usepackage{color}
\usepackage{amsfonts,mathrsfs,bbm,amsmath,stmaryrd,amssymb,pifont}%
\usepackage{amscd,graphicx,array,dsfont,texdraw,tikz}%
\usetikzlibrary{arrows,decorations.markings,shapes,shadows,fadings}
\usepackage{bbm,mathbbol,mathrsfs}%
\usepackage{array}
\usepackage{enumerate}
\usepackage{times}

\def\BibTeX{{\rm B\kern-.05em{\sc i\kern-.025em b}\kern-.08em
    T\kern-.1667em\lower.7ex\hbox{E}\kern-.125emX}}

\newtheorem{lemma}{Lemma}
\newtheorem{theorem}{Theorem}
\newtheorem{corollary}{Corollary}
\newtheorem{remark}{Remark}[section]
\newtheorem{assumption}{Assumption}

\newtheorem{problem}{Problem}

\def\diag{\mathrm{diag}}

\def\dref#1{(\ref{#1})}

\begin{document}
\title{Solving specified-time distributed optimization problem via sampled-data-based algorithm}
\author{Jialing Zhou, Yuezu Lv, Changyun Wen, and Guanghui Wen
	\thanks{Jialing Zhou is with the School of Automation, Nanjing University of Science and Technology, Nanjing 210094, China (e-mail: jialingz@njust.edu.cn).}
	\thanks{Yuezu Lv and Guanghui Wen are with the School of Mathematics, Southeast University, Nanjing 211189, China (e-mail: yzlv@seu.edu.cn, wenguanghui@gmail.com).}
	\thanks{Changyun Wen is with the School of Electrical and Electronic Engineering, Nanyang Technological University, Nanyang Avenue 639798, Singapore (e-mail: ecywen@ntu.edu.sg).}
}

\maketitle

\begin{abstract}
Despite significant advances on distributed continuous-time optimization of multi-agent networks, there is still lack of an efficient algorithm to achieve the goal of distributed optimization at a pre-specified time.
Herein, we design a specified-time distributed optimization algorithm for connected agents with directed topologies to collectively minimize the sum of individual objective functions subject to an equality constraint. With the designed algorithm, the settling time of distributed optimization can be exactly predefined. The specified selection of such a settling time is independent of not only the initial conditions of agents, but also the algorithm parameters and the communication topologies. Furthermore, the proposed algorithm can realize specified-time optimization by exchanging information among neighbours only at discrete sampling instants and thus reduces the communication burden. In addition, the equality constraint is always satisfied during the whole process, which makes the proposed algorithm applicable to online solving distributed optimization problems such as economic dispatch. For the special case of undirected communication topologies, a reduced-order algorithm is also designed. Finally, the effectiveness of the theoretical analysis is justified by numerical simulations.
\end{abstract}


\section{Introduction}

Over the past few years, we have witnessed tremendous progress on distributed optimization of multi-agent networks (MANs) \cite{survey2018,survey-yangtao}, owing partly to its potential applications in various engineering fields such as distributed sensor networks, smart grid, and multiple robots. The objective of distributed optimization over MANs is to minimize the group objective function by designing only local information-based algorithm for each agent \cite{survey2018}. 

In existing literature, various  algorithms have been proposed to solve distributed optimization problems in MANs \cite{subgradient2009,Primal-dual2016,push-sum-yangtao,gfh-cluster-consensus,gfh-cluster-dispatch,push-pull-pushi,zero-gradient-sum2012,zero-gradient-sum2016renwei,continous-pi}.
Among these, considerable quantity of algorithms are developed under iterative calculation scheme, such as the distributed subgradient descent algorithm \cite{subgradient2009}, the distributed primal-dual algorithm \cite{Primal-dual2016}, the push-sum based algorithms \cite{push-sum-yangtao}, the push-pull based algorithm \cite{push-pull-pushi}, the distributed hierarchical algorithm \cite{gfh-cluster-consensus,gfh-cluster-dispatch}, etc. Noticing that many practical systems operate in continuous-time, the other category of algorithms are developed under continuous-time updating scheme, such as the distributed continuous-time PI algorithm \cite{continous-pi}, the zero-gradient-sum algorithm \cite{zero-gradient-sum2012,zero-gradient-sum2016renwei}, just to name a few. 

All the above-mentioned discrete-time and continuous-time algorithms  converge to the optimal solutions asymptotically. However, in some practical applications such as energy resource allocation, fast and time-critical responses are required
\cite{survey-yangtao}. Therefore, developing distributed algorithms which are able to reach optimal solutions in finite time is highly desirable.
There are several available results about distributed finite-time optimization,
where the discrete-time algorithms are developed mainly based on the finite-time consensus technique \cite{discrete-time-finite-time1,discrete-time-finite-time2}, and the continuous-time algorithms are designed using the finite-time stability theory \cite{finite-time}.
It is worth noticing that the aforementioned discrete-time algorithms are only applicable to quadratic cost functions; and the upper bounds of the settling times of the aforementioned continuous-time algorithms depend on initial conditions of the MANs.
In \cite{fixed-time-ningboda,ChenLiAuto2018ConDis}, the authors step further to develop distributed fixed-time optimization algorithms which enable the upper bounds of settling times independent of initial conditions, which yields the advantage that the upper bounds of settling times can be preassigned even if the information of initial conditions is unavailable to the agents.

From the above discussions, it can be observed that convergence rate and settling time are significant performance indexes to be concerned about.
This motivates us to develop specified-time distributed optimization of MANs. Note that such an optimization problem is quite different from fixed-time optimization. The latter only requires pre-estimation of the upper bound of the settling time, which is conservative to certain degree, while the former requires precisely pre-determination of the settling time. In this paper, we aim to solve the specified-time distributed convex optimization subject to equality constraint under sampled-data communications with general directed topologies. A novel design strategy is established, under which the equality constraint is  satisfied during the whole process, and the settling time is independent of not only the initial conditions, but also the algorithm parameters and communication topologies. For the special case of undirected topologies, a reduced-order algorithm is also designed.
The main contributions of this paper are at least three folds.
Firstly,  different form existing results on finite-time and fixed-time optimization \cite{finite-time,fixed-time-ningboda,ChenLiAuto2018ConDis}, the settling time in this paper is no longer conservative but can be precisely preassigned according to practical demand.
Moreover, compared to existing results on fixed-time optimization \cite{ChenLiAuto2018ConDis,fixed-time-ningboda} where the bound of settling time depends on algorithm parameters and communication topologies, the settling time in this paper can be determined independently, which is more convenient for real-world applications.
Secondly, unlike related results in \cite{finite-time,fixed-time-ningboda,ChenLiAuto2018ConDis}, the proposed algorithm requires the information from neighbouring nodes only  at discrete sampling instants rather than for the whole time interval, which reduces communication burden and thus the communication cost of the MANs.
Thirdly, different from \cite{ChenLiAuto2018ConDis} only considering undirected topologies, the design structure established in this paper ensures that the equality constraint is satisfied during the whole process even if the communication topology is directed, which makes the proposed algorithm able to online solving distributed optimization problems such as economic dispatch for general directed multi-agent networks.

The remainder of the paper is  {summarized} as follows. In Section \ref{SectionTwo},  {preliminaries are  provided and the problem is formulated.}  {Section \ref{SectionThree} presents the specified-time distributed optimization algorithm design for directed MANs, and}
the special case of undirected topologies is investigated in Section \ref{SectionFour}.  {In Section \ref{SectionFive},
	numerical simulations are conducted to validate the theoretical analysis, and the results show the} advantages of the specified-time convergent algorithms.  {Finally, the paper is concluded in} Section \ref{SectionSix}. 

\textbf{Notation}.  {Symbol $\otimes$ denotes the Kronecker product.  Symbol $\|\cdot\|$  represents the Euclidian norm.}
Symbols $\mathbb{R}$, $\mathbb{N}$ and $\mathbb{N}^+$  respectively denote the sets of real numbers, non-negative integers and positive integers. $\bm{1}$ denotes a column vector with
all entries being $1$. $\nabla f(\cdot)$ and
$\nabla^2f (\cdot)$  represent the gradient and the Hessian matrix of
the function $f (\cdot)$, respectively. For a real and symmetric matrix $A$, write $A\succeq 0$ $(A\preceq 0)$ if $A$ is
positive (negative) semi-definite.

\section{Preliminaries and problem formulation}\label{SectionTwo}
\subsection{Preliminaries}
The underlying interaction graph of a MAN  containing $n$ agents can be represented by a directed graph $\mathcal{G}(\mathcal{V},\mathcal{E})$ with the node (agent) set $\mathcal{V}{=}\{1,\cdots,n\}$ and the edge (communication link) set $\mathcal{E}{\subseteq} \mathcal{V}{\times} \mathcal{V}$. An edge $(i,j){\in}\mathcal{E}$ represents that node $j$ can receive information from node $i$. A path from node $i_1$ to node $i_l$ is a sequence of ordered edges of the form $(i_k, i_{k+1}), k{=} 1,{\cdots},l{-}1$. Graph $\mathcal{G}$ is called strongly connected if  for any node, there exist paths to all
other nodes.
The adjacency matrix $\mathcal{A} {=}[a_{ij} ]{\in} \mathbb{R}^{n{\times}n}$ of graph $\mathcal{G}$ is defined as $a_{ii}{=} 0$,
$a_{ij}{=} 1$ if $(j,i){\in}\mathcal{E}$  and $0$ otherwise.
The Laplacian matrix $\mathcal{L}{=}[l_{ij} ]{\in} \mathbb{R}^{n{\times}n}$ of graph $\mathcal{G}$ is defined as $l_{ii}{=} \sum_{j=1}^n a_{ij}$ and
$l_{ij}{=}{-} a_{ij},i\neq j$.
Define  the in-degree matrix  $\mathcal{D}{=}\diag\{d^{\text{in}}_1,{\cdots},d^{\text{in}}_n\}$ with $d_{i}^{\text{in}}{=}\sum_{j=1}^n a_{ij}$ denoting the in-degree of agent $i$. Then, the Laplacian matrix $\mathcal{L}$ can be rewritten as $\mathcal{L}{=}\mathcal{D}{-}\mathcal{A}$. Obviously, $\mathcal{L}\bm{1}{=}0$. Moreover, for a strongly connected graph, zero is a simple eigenvalue of the Laplacian matrix and all the other
eigenvalues lie in the right plane \cite{renwei2007Information-consensus}.

Similarly, define  $\mathcal{L_O}{=} \mathcal{D_O}{-}\mathcal{A}$, where  $\mathcal{D_O}{=}\diag\{d^{\text{out}}_1,{\cdots},d^{\text{out}}_n\}$ with $d_{i}^{\text{out}}{=}\sum_{j=1}^n a_{ji}$ denoting the out-degree of agent $i$. Obviously, $\bm{1}^T\mathcal{L_O}{=}0$.  Note that $\mathcal{L_O}^T$ can be viewed as the Laplacian matrix of graph $\mathcal{G}'{=}(\mathcal{V},\mathcal{E}')$, where $\mathcal{E}'{=}\{(j,i)|(i,j){\in}\mathcal{E}\}$. Thus, graph $\mathcal{G}'$ is strongly connected if and only if graph $\mathcal{G}$ is strongly connected.
 Clearly, if $\mathcal{G}$ is undirected,  $\mathcal{L_O}{=}\mathcal{L}$.

The following lemma is commonly known in matrix theory.
\begin{lemma}\label{lemma.matrix}
For a symmetric matrix $A{\in}\mathbb{R}^{n\times n}$, denote the eigenvalues of $A$ by $\lambda_1,{\cdots},\lambda_n$ satisfying $\lambda_1{\leq}{\cdots}{\leq}\lambda_n$. Then,
$
y^T A y{\geq} \lambda_i y^Ty, ~\forall y{\in}\{y|y{\perp} v_j, j{=}1,{\cdots}, i{-}1\},~ i{=}1,{\cdots},n,
$
where $v_j$ is the eigenvector of $A$ associated with the eigenvalue $\lambda_j$.
\end{lemma}



\subsection{Problem formulation}
The MAN is assumed to contain $n$ agents where the state of agent $i$ $(i{=}1,{\cdots},n)$ at time $t$ is denoted as $x_{i}(t){\in} \mathbb{R}$. Define $\bm{x}(t){=}[x_1(t),\cdots, x_n(t)]^T$. Herein, the optimization objective is to minimize the sum of $n$ local objective functions $f_{i}(x_{i})$ at a pre-specified time $T_{c}$ $(T_{c}{>}0)$ subject to a global equality constraint $\sum_{i=1}^{n}x_{i}(t){=}\sum_{i=1}^{n}x_{i}(0){=}C$ with $x_i(0)$ being the initial value of agent $i$ and $C$ being a given real scalar.

Denote $\bm{x}^{*}{\in}\mathbb{R}^n$ the minimum of the following time-invariant
optimization problem
\begin{equation*}
\begin{aligned}
\min_{\bm{x}} f(\bm{x})&=\min_{\bm{x}} \sum_{i=1}^{n}f_{i}(x_{i}), \\
\mathrm{subject}~~\mathrm{to}~& \sum_{i=1}^{n}x_{i}=\sum_{i=1}^{n}x_{i}(0)=C.
\end{aligned}
\end{equation*}
The following assumption ensures the  existence and uniqueness of  $\bm{x}^*$.
\begin{assumption}\label{AssumptionLip}
	The objective function $f_{i}(\cdot)$ $(i{=}1,{\cdots},n)$
	is strongly convex such that $\nabla^{2}f_{i}(\cdot){\geq} l_0$ with $l_0$ a given positive real number.
	Moreover, the gradient of function $f(\cdot)$ is
	Lipschitz continuous, i.e., there exists a positive constant $l$, such that $\|\nabla f(\bm{a}){-}\nabla f(\bm{b})\|{\leq}  l \|\bm{a}{-}\bm{b}\|$ for any $\bm{a},\bm{b}{\in}\mathbb{R}^n$.
\end{assumption}

Moreover, applying the method of Lagrange multiplier, one can obtain that $\bm{x}^*$ satisfies $\nabla f(\bm{x}^*){=}{-}\lambda^*\bm{1}$ with $\lambda^*{\in}\mathbb{R}$ denoting the optimal Lagrange multiplier, which further implies that
\begin{equation}
\mathcal{L}\nabla f(\bm{x}^*)=0,\mathcal{L_O}^T\nabla f(\bm{x}^*){=}0.
\end{equation}

Then, the problem investigated in this paper can be stated as follows.
\begin{problem}\label{problem}
	For the MAN with underlying interaction graph $\mathcal{G}$, design a distributed algorithm such that the MAN state $\bm{x}(t)$ converges to $\bm{x}^*$ at the pre-specified time $T_c$ while satisfies
	\begin{equation}\label{eq.problem.constaint}
		\sum_{i=1}^{n}x_{i}(t)=\sum_{i=1}^{n}x_{i}(0)=C, \forall t\in [0,T_c].
	\end{equation}
\end{problem}

\section{Specified-time distributed optimization for directed MANs}\label{SectionThree}

In this section, we consider the distributed algorithm design under general directed graph.
To solve the specified-time distributed optimization problem with the global equality constraint, we design the following continuous-time algorithm for agent $i{=}1,{\cdots},n$ based on sampled-data communication:
\begin{equation}\label{eq.law.x.i}
\begin{aligned}
x_{i}(t)=&x_{i}(0)-d_i^{\text{out}}\xi_{i}(t_k)+\sum_{j=1}^n a_{ij}\xi_{j}(t_k),\\
\dot{\xi}_{i}(t)=&\frac{\beta}{t_{k+1}-t_{k}}\big(d_i^{\text{out}}\psi_{ii}(t_k)-\sum_{j=1}^Na_{ji}\psi_{ij}(t_k)\big),\\
\dot\psi_{im}(t)=&-\frac{1}{(t_{k+1}-t_k)(d_i^{\text{in}}+a_{im})}\bigg(\sum_{j=1}^{n}a_{ij}(\psi_{im}(t_k)\\
&-\psi_{jm}(t_k))+a_{im}\big(\psi_{im}(t_k)-\frac{df_m}{dx_m}(x_m(t_k))\big)\bigg),\\
&\forall t{\in}[t_k,t_{k+1}), ~k{\in}\mathbb{N},~m=1,{\cdots},n,
\end{aligned}
\end{equation}
where $\xi_i$ is an auxiliary variable of agent $i$, $\psi_{im}$ denotes the estimation of $\frac{df_m(x_m)}{dx_m}$ by agent $i$, $\beta$ is a small positive constant to be determined later, $\xi_i(0)=0$,  $\psi_{im}(0)=0$, and $t_{k}$ denotes the $k$-th sampled instant. To achieve convergence at the specified time $T_c$, the following sampled instants are designed:
\begin{equation} \label{SamplingTimeSec2}
t_{k}=\left\{ \begin{aligned}
&\sum_{l=1}^{k}T_{l}, \quad k{\in}\mathbb{N}^+, \\
& 0, \quad\;\;\;\;\;\; \;\;  k=0.
\end{aligned} \right.
\end{equation}
where
\begin{equation}\label{SamplingTimeIntervalSec2}
T_{k}=\frac{6}{(\pi k)^2}T_{c}, \quad k{\in}\mathbb{N}^+.
\end{equation}
Define  $\bm{\psi}_i(t){=}[\psi_{i1}(t),{\cdots},\psi_{in}(t)]^T{\in}\mathbb{R}^n$. Then,  $\bm{\psi}_i(t)$ can be viewed as the estimation of gradient $\nabla f(\bm{x}(t))$ by agent $i$.
Under the proposed algrithm (\ref{eq.law.x.i}), agent $i$ requires the information of $\xi_j$, $\frac{df_j(x_j)}{dx_j}$, and  $\bm{\psi}_{j}$ from its in-neighbor  $j$
at sampling time instants $t_k,t{\in}\mathbb{N}$.
Note that the second equation of (\ref{eq.law.x.i}) does not require information exchange among neighboring agents, since  $\psi_{i1},\cdots,\psi_{in}$ are local information of agent $i$.

\begin{lemma}\label{lemma.equality.constraint}
	Under the proposed algorithm (\ref{eq.law.x.i}), the equality constraint (\ref{eq.problem.constaint}) in Problem \ref{problem} is always satisfied.
\end{lemma}
\begin{proof}
	Define $\bm{\xi}(t){=}[\xi_1(t),\xi_2(t),{\cdots}, \xi_n(t)]^T{\in} \mathbb{R}^{n}$. Then, the compact form of the agent states can be written as
	\begin{equation}
	\bm{x}(t)=\bm{x}(0)-\mathcal{L_O}\bm{\xi}(t_k),\quad t{\in}[t_k,t_{k+1}),\;k{\in}\mathbb{N}.
	\end{equation}
	Noting that $\bm{1}^T\mathcal{L_O}{=}0$, one has
	$\bm{1}^T\bm{x}(t){=}\bm{1}^T\bm{x}(0){=}C,$ $\forall t{\geq}0$.
\end{proof}

\begin{remark}\label{rem1}
In (\ref{eq.law.x.i}), to meet the global equality constraint during the whole process, the auxiliary variable $\bm{\xi}(t)$ is introduced and the design structure $\bm{x}(t){=}\bm{x}(0){-}\mathcal{L_O}\bm{\xi}(t_k)$ is established. Under this framework,
the problem of searching $\bm{x}^*$ can be reformulated as the problem of finding the minimum of the composite function $f\left(\bm{x}(0){-}\mathcal{L_O}^T\bm{\xi}\right)$ about $\bm{\xi}$. To achieve this goal through the gradient descent method, the information $\mathcal{L_O}^T\nabla f(\bm{x})$ is required in the design of $\bm{\xi}$, i.e., agent $i$ requires the value $\frac{df_j}{dx_j}(x_j)$ of its out-neighbor $j$, which however, cannot be obtained, since an agent can only receive information from its in-neighbors. To overcome this difficulty, in (\ref{eq.law.x.i}), observers are designed for each agent to estimate the gradient $\nabla f(\bm{x})$, such that the auxiliary variable of each agent can update based on its local information. 
\end{remark}

Integrating (\ref{eq.law.x.i}) from $t_k$ to $t_{k+1}$ yields
\begin{equation}\label{eq.law.xi.i}
\begin{aligned}
{\xi}_{i}(t_{k+1})&={\xi}_{i}(t_{k})+{\beta}\big(d_i^{\text{out}}\psi_{ii}(t_k)-\sum_{j=1}^Na_{ji}\psi_{ij}(t_k)\big),\\
\psi_{ij}(t_{k+1})&=\psi_{ij}(t_{k})-\frac{1}{d_i^{\text{in}}+a_{ij}}\bigg(\sum_{i=m}^{n}a_{im}(\psi_{ij}(t_k)\\
&-\psi_{mj}(t_k))+a_{ij}\big(\psi_{ij}(t_k)-\frac{df_j}{dx_j}(x_j(t_k))\big)\bigg).
\end{aligned}
\end{equation}

Define $\bm{\xi}{=}[\xi_1,{\cdots},\xi_n]^T{\in}\mathbb{R}^n$, $\bm{\psi}_i{=}[\psi_{i1},{\cdots},\psi_{in}]^T{\in}\mathbb{R}^n$, $\bm{\psi}{=}[\bm{\psi}_1^T,{\cdots},\bm{\psi}_n^T]^T{\in}\mathbb{R}^{n^2}$.
Then, (\ref{eq.law.xi.i}) can be rewritten as
\begin{equation}\label{eq.xi.t_k+1.t_k.1}
\begin{aligned}
\bm{\xi}(t_{k+1})=&\bm{\xi}(t_{k})+{\beta}\hat{\mathcal{L}}_0\bm{\psi}(t_k),\\
\bm{\psi}(t_{k+1})
{=}&\bm{\psi}(t_k){-}\Gamma(\mathcal{L}{\otimes} I_{n}+A_d)\big(\bm{\psi}(t_k)\\
&{-}\bm{1}_{n}{\otimes}\nabla f(\bm{x}(t_k))\big),
\end{aligned}
\end{equation}
where $\hat{\mathcal{L}}_0=\diag\{(\mathcal{L_O}^T)_1,\cdots,(\mathcal{L_O}^T)_n\}$ with $(\mathcal{L_O}^T)_i$ denoting the $i$th row of the matrix $\mathcal{L_O}^T$, and
\begin{equation*}
\begin{aligned}
\Gamma=&\diag\{\frac{1}{d_{1}^{\text{in}}+a_{11}},\cdots,\frac{1}{d_{1}^{\text{in}}+a_{1n}},\frac{1}{d_{2}^{\text{in}}+a_{21}},\\
&{\cdots},\frac{1}{d_{2}^{\text{in}}+a_{2n}},{\cdots},\frac{1}{d_{n}^{\text{in}}+a_{nn}}
\}{\in}\mathbb{R}^{n^2\times n^2},\\
A_d=&\diag\{a_{11},{\cdots},a_{1n},a_{21},{\cdots},a_{2n},{\cdots},a_{nn}\}{\in}\mathbb{R}^{n^2\times n^2}.
\end{aligned}
\end{equation*}
Define the estimation error
\begin{equation}
\bm{e_\psi}(t)=\bm{\psi}(t)-\bm{1}_n{\otimes}\nabla f(\bm{x}(t)).
\end{equation}
Then, the iteration of $\bm{e_\psi}$ can be given by
\begin{equation}
\begin{aligned}
&\bm{e_{\psi}}(t_{k+1})\\
=&\bm{\psi}(t_k)-\Gamma(\mathcal{L}{\otimes} I_{n}+A_d)\bm{e_{\psi}}(t_{k})
-\bm{1}_n{\otimes}\nabla f(\bm{x}(t_{k+1}))\\
=&\mathcal{M}\bm{e_\psi}(t_k)-\bm{1}_n{\otimes}\big(\nabla f(\bm{x}(t_{k+1}))-\nabla f(\bm{x}(t_{k}))\big),\\
\end{aligned}
\end{equation}
where $\mathcal{M}=I_{n^2}-\Gamma(\mathcal{L}{\otimes} I_{n}+A_d)$.
If the graph $\mathcal{G}$ is strongly connected, one can derive that $\mathcal{M}$ is a Schur matrix, and there exists a symmetric positive definite matrix $W$, such that $\mathcal{M}^TW\mathcal{M}-W=-I_{n^2}$.

By the definition of $\bm{e_\psi}$, we have
\begin{equation}
\begin{aligned}
&\bm{\xi}(t_{k+1})\\
=&\bm{\xi}(t_{k})+{\beta}\hat{\mathcal{L}}_0\Big(\bm{1}_n{\otimes}\nabla f(\bm{x}(t_k))+\bm{e_\psi}(t_k)\Big)\\
=&\bm{\xi}(t_{k})+{\beta}{\mathcal{L_O}}^T\nabla f(\bm{x}(t_k))+{\beta}\hat{\mathcal{L}}_0\bm{e_\psi}(t_k)\\
\end{aligned}
\end{equation}
where we have used the fact that $\hat{\mathcal{L}}_0(\bm{1}_n{\otimes}\nabla f(\bm{x}))=\mathcal{L_O}^T\nabla f(\bm{x})$ to obtain the last equality.

Now we are ready to present the main result of this section.

\begin{theorem}\label{TheoremTwoMMatrix}\rm
	Suppose that Assumption \ref{AssumptionLip} holds, and graph $\mathcal{G}$ is strongly connected. Then, the distributed algorithm (\ref{eq.law.x.i}), (\ref{SamplingTimeSec2}) and (\ref{SamplingTimeIntervalSec2})
	solves  Problem \ref{problem} if
	\begin{equation}\label{eq.beta}
	\begin{aligned}
	\beta\leq\min\{&\frac{1}{2\|\hat{\mathcal{L}}_0\|^2\big(1+4{l}^2b\|\mathcal{L_O}\|^2+2{{l}}\|\mathcal{L_O}\|^2\big)},\\
	&\frac{1}{4\big(2{l}^2b\|\mathcal{L_O}\|^2+{{l}}\|\mathcal{L_O}\|^2\big)},1\},
	\end{aligned}
	\end{equation}
	where $b{=}\big(2\|\mathcal{M}^TW\|^2+\|W\|\big)n$.
\end{theorem}
\begin{proof}
	Consider the Lyapunov function
	\begin{equation}\label{lya1}
	\begin{aligned}
	V(t)&=\bm{e_\psi}^T(t)W\bm{e_\psi}(t)+f(\bm{x}(t))-f(\bm{x}^*).\\
	\end{aligned}
	\end{equation}
Then,
	\begin{equation}\label{dlya11}
	\begin{aligned}
	&V(t_{k+1})-V(t_{k})\\
	=&\bm{e_\psi}^T(t_{k+1})W\bm{e_\psi}(t_{k+1})-\bm{e_\psi}^T(t_k)W\bm{e_\psi}(t_k)\\
	&+f(\bm{x}(t_{k+1}))-f(\bm{x}(t_k)).
	\end{aligned}
	\end{equation}
	
	It is not difficult to calculate that
	\begin{equation}\label{eq.depsi1}
	\begin{aligned}
	&\bm{e_\psi}^T(t_{k+1})W\bm{e_\psi}(t_{k+1})-\bm{e_\psi}^T(t_{k})W\bm{e_\psi}(t_{k})\\
	=&\bm{e_\psi}^T(t_k)(\mathcal{M}^TW\mathcal{M}-W)\bm{e_\psi}(t_k)\\
	&-2\bm{e_\psi}^T(t_k)\mathcal{M}^TW(\bm{1}_n{\otimes}(\nabla f(\bm{x}(t_{k+1}))-\nabla f(\bm{x}(t_{k}))))\\
	&+(\bm{1}_n{\otimes}(\nabla f(\bm{x}(t_{k+1}))-\nabla f(\bm{x}(t_{k}))))^TW\\
	&\times(\bm{1}_n{\otimes}(\nabla f(\bm{x}(t_{k+1}))-\nabla f(\bm{x}(t_{k}))))\\
	\leq &-\|\bm{e_\psi}(t_k)\|^2\\
	&+2\|\bm{e_\psi}(t_k)\|\|\mathcal{M}^TW\|\sqrt{n}\|\nabla f(\bm{x}(t_{k+1}))-\nabla f(\bm{x}(t_{k}))\|\\
	&+\|W\|n\|\nabla f(\bm{x}(t_{k+1}))-\nabla f(\bm{x}(t_{k}))\|^2\\
	\leq &-\frac{1}{2}\|\bm{e_\psi}(t_k)\|^2+b\|\nabla f(\bm{x}(t_{k+1}))-\nabla f(\bm{x}(t_{k}))\|^2,
	\end{aligned}
	\end{equation}
	where we have used
	\begin{equation*}
	\begin{aligned}
	&2\|\bm{e_\psi}(t_k)\|\|\mathcal{M}^TW\|\sqrt{n}\|\nabla f(\bm{x}(t_{k+1}))-\nabla f(\bm{x}(t_{k}))\|\\
	\leq&\frac{1}{2}\|\bm{e_\psi}(t_k)\|^2+2\|\mathcal{M}^TW\|^2n\|\nabla f(\bm{x}(t_{k+1}))-\nabla f(\bm{x}(t_{k}))\|^2
	\end{aligned}
	\end{equation*}
	to obtain the last inequality.
Under Assumption \ref{AssumptionLip}, we have
	\begin{equation}\label{eq.gradf}
	\begin{aligned}
	&\|\nabla f(\bm{x}(t_{k+1}))-\nabla f(\bm{x}(t_{k}))\|^2\\
	\leq&{l}^2\|\bm{x}(t_{k+1})-\bm{x}(t_{k})\|^2\\
	\leq&{l}^2\|\mathcal{L_O}\|^2\|\bm{\xi}(t_{k+1})-\bm{\xi}(t_{k})\|^2\\
	=&\beta^2{l}^2\|\mathcal{L_O}\|^2\left\|{\mathcal{L_O}}^T\nabla f(\bm{x}(t_k))-\hat{\mathcal{L}}_0\bm{e_\psi}(t_k)\right\|^2\\
	\leq&2\beta^2{l}^2\|\mathcal{L_O}\|^2( \|{\mathcal{L_O}}^T\nabla f(\bm{x}(t_k))\|^2+\|\hat{\mathcal{L}}_0\|^2\|\bm{e_\psi}(t_k)\|^2).
	\end{aligned}
	\end{equation}
	Substituting (\ref{eq.gradf}) into (\ref{eq.depsi1}) yields
	\begin{equation}\label{eq.depsi2}
	\begin{aligned}
	&\bm{e_\psi}^T(t_{k+1})W\bm{e_\psi}(t_{k+1})-\bm{e_\psi}^T(t_{k})W\bm{e_\psi}(t_{k})\\
	\leq&-\frac{1}{2}\|\bm{e_\psi}(t_k)\|^2+2\beta^2{l}^2b\|\mathcal{L_O}\|^2\big( \|{\mathcal{L_O}}^T\nabla f(\bm{x}(t_k))\|^2\\
	&+\|\hat{\mathcal{L}}_0\|^2\|\bm{e_\psi}(t_k)\|^2\big).
	\end{aligned}
	\end{equation}

	Assumption \ref{AssumptionLip} implies that $\nabla^2 f(\bm{x})\preceq {l} I,\forall \bm{x}{\in}\mathbb{R}^n$.
	Using this fact, we can perform a quadratic expansion of $f(\bm{x}(t_{k+1}))$ around $f(\bm{x}(t_k))$ and obtain the following inequality :
	\begin{equation}\label{eq.df1}
	\begin{aligned}
	&f(\bm{x}(t_{k+1}))\\
	\leq & f(\bm{x}(t_k))+\nabla^Tf(\bm{x}(t_k))(\bm{x}(t_{k+1})-\bm{x}(t_k))\\
	&+\frac{{l}}{2}(\bm{x}(t_{k+1})-\bm{x}(t_k))^T(\bm{x}(t_{k+1})-\bm{x}(t_k))\\
	=& f(\bm{x}(t_k))-\beta \nabla^T f(\bm{x}(t_k)){\mathcal{L_O}}\left({\mathcal{L_O}}^T\nabla f(\bm{x}(t_k))+\hat{\mathcal{L}}_0\bm{e_\psi}(t_k)\right)\\
	&+\frac{{l}\beta^2}{2}\left\|\mathcal{L_O}\left({\mathcal{L_O}}^T\nabla f(\bm{x}(t_k))+\hat{\mathcal{L}}_0\bm{e_\psi}(t_k)\right)\right\|^2\\
	\leq & f(\bm{x}(t_k))-{\beta}\left\|{\mathcal{L_O}}^T\nabla f(\bm{x}(t_k))\right\|^2\\
	&+{\beta}\left\|{\mathcal{L_O}}^T\nabla f(\bm{x}(t_k))\right\|\left\|\hat{\mathcal{L}}_0\bm{e_\psi}(t_k)\right\|\\
	&+\frac{{l}\beta^2}{2}\|\mathcal{L_O}\|^2 \left\|{\mathcal{L_O}}^T\nabla f(\bm{x}(t_k))+\hat{\mathcal{L}}_0\bm{e_\psi}(t_k)\right\|^2\\
	\leq &f(\bm{x}(t_k))-\frac{\beta}{2}\|{\mathcal{L_O}}^T\nabla f(\bm{x}(t_k))\|^2+\frac{\beta}{2}\|\hat{\mathcal{L}}_0\|^2\|\bm{e_\psi}(t_k)\|^2\\
	&+{{l}\beta^2}\|\mathcal{L_O}\|^2\left( \|{\mathcal{L_O}}^T\nabla f(\bm{x}(t_k))\|^2+\|\hat{\mathcal{L}}_0\|^2\|\bm{e_\psi}(t_k)\|^2\right),\\
	\end{aligned}
	\end{equation}
	where we have used
	\begin{equation*}
	\begin{aligned}
	&{\beta}\|{\mathcal{L_O}}^T\nabla f(\bm{x}(t_k))\|\|\hat{\mathcal{L}}_0\bm{e_\psi}(t_k)\|\\
	\leq&\frac{\beta}{2}\|{\mathcal{L_O}}^T\nabla f(\bm{x}(t_k))\|^2+\frac{\beta}{2}\|\hat{\mathcal{L}}_0\|^2\|\bm{e_\psi}(t_k)\|^2
	\end{aligned}
	\end{equation*}
	to obtain the last inequality.
	
	Subsitituting (\ref{eq.depsi2}) and (\ref{eq.df1}) into (\ref{dlya11}) yields
	\begin{equation}
	\begin{aligned}
	&V(t_{k+1})-V(t_{k})\\
	\leq&-\frac{1}{2}\|\bm{e_\psi}(t_k)\|^2-\frac{\beta}{2}\|{\mathcal{L_O}}^T\nabla f(\bm{x}(t_k))\|^2\\
	&+\beta\big(\frac{1}{2}+2\beta{l}^2b\|\mathcal{L_O}\|^2+{{l}\beta}\|\mathcal{L_O}\|^2\big)\|\hat{\mathcal{L}}_0\|^2\|\bm{e_\psi}(t_k)\|^2\\
	&+\beta^2\big(2{l}^2b\|\mathcal{L_O}\|^2+{{l}}\|\mathcal{L_O}\|^2\big) \|{\mathcal{L_O}}^T\nabla f(\bm{x}(t_k))\|^2.
	\end{aligned}
	\end{equation}
	Noting that $\beta$ satisfies \dref{eq.beta},
	we can obtain  that
	\begin{equation}\label{eq.dV}
	\begin{aligned}
	&V(t_{k+1})-V(t_{k})\\
	\leq&-\frac{1}{4}\|\bm{e_\psi}(t_k)\|^2-\frac{\beta}{4}\|{\mathcal{L_O}}^T\nabla f(\bm{x}(t_k))\|^2,
	\end{aligned}
	\end{equation}
which shows that $V(t_{k+1}){\leq}V(t_k),~\forall k{\in}\mathbb{N}$.
	Next, we will analyse the convergence rate of the Lyapunov function.
	From the convex property of function $f(\bm{x})$, one has
	\begin{equation}\label{convex_ineq}
	\begin{aligned}
	f(\bm{x}(t_k)){\leq}&f(\bm{x}^*){+}\nabla f(\bm{x}(t_k))^T(\bm{x}(t_k){-}\bm{x}^*)\\
	{=}&f(\bm{x}^*){-}\nabla f(\bm{x}(t_k))^T\mathcal{L_O}(\bm{\xi}(t_k){-}\bm{\xi}^*),
	\end{aligned}
	\end{equation}
	where $\bm{\xi}^*{\in}\mathbb{R}^n$ is a constant vector which satisfies $\bm{x}^*{=}\bm{x}(0){-}\mathcal{L_O}\bm{\xi}^*$.
	Define $\bm{v}_1$ the right unit eigenvector of $\mathcal{L_O}$ corresponding to the zero eigenvalue, i.e., $\mathcal{L_O}\bm{v}_1{=}0$ and $\bm{v}_1^T\bm{v}_1{=}1$. Define a set of unit vectors $\bm{v}_2,{\cdots},\bm{v}_n$ which form a normal orthonormal basis together with $\bm{v}_1$.
	Then, there exist a set of real numbers $b_1,b_2,\cdots,b_n$ at $t_k$ such that
	\begin{equation}\label{eq.b}
	\bm{\xi}(t_k){-}\bm{\xi}^*=\sum_{i=1}^{n}b_i(t_k)\bm{v}_i.
	\end{equation}
Substituting (\ref{eq.b}) into \dref{convex_ineq} yields
	\begin{equation}\label{eq.fxk-f*.leq1}
	\begin{aligned}
	&f(\bm{x}(t_k))-f(\bm{x}^*)\\
	{\leq}&{-}\nabla^T f(\bm{x}(t_k))\mathcal{L_O}\sum_{i=2}^{n}b_i(t_k)\bm{v}_i\\
	\leq&\left\|\nabla^T f(\bm{x}(t_k))\mathcal{L_O}\right\|\left\|\sum_{i=2}^{n}b_i(t_k)\bm{v}_i\right\|.
	\end{aligned}
	\end{equation}
	Under Assumption \ref{AssumptionLip}, one can derive that
	\begin{equation*}
	\begin{aligned}
	&f(\bm{x}(t_k))-f(\bm{x}^*)\\
	\geq& \nabla^T f(\bm{x}^*)(\bm{x}(t_k)-\bm{x}^*)+ \frac{l_0}{2}\|\bm{x}(t_k)-\bm{x}^*\|^2.
	\end{aligned}
	\end{equation*}
	Note that
	\begin{equation*}
	\nabla^T f(\bm{x}^*)(\bm{x}(t_k)-\bm{x}^*)=-\nabla^T f(\bm{x}^*)\mathcal{L_O}(\bm{\xi}(t_k)-\bm{\xi}^*)=0.
	\end{equation*}
	Then,
	\begin{equation}\label{eq.fxk-f*.geq1}
	\begin{aligned}
	&f(\bm{x}(t_k))-f(\bm{x}^*)\\
	\geq& \frac{l_0}{2}\|\bm{x}(t_k)-\bm{x}^*\|^2\\
	=& \frac{l_0}{2}\left\|\mathcal{L_O}\sum_{i=2}^{n}b_i(t_k)\bm{v}_i\right\|^2\\
	\geq& \frac{l_0}{2}\lambda_2(\mathcal{L_O}^T\mathcal{L_O})\left\|\sum_{i=2}^{n}b_i(t_k)\bm{v}_i\right\|^2,
	\end{aligned}
	\end{equation}
	where  $\lambda_{2}(\mathcal{L_O}^T\mathcal{L_O})$ denotes the minimal non-zero eigenvalue of $\mathcal{L_O}^T\mathcal{L_O}$, and the last inequality is obtained by using Lemma \ref{lemma.matrix}.
	Combining (\ref{eq.fxk-f*.leq1}) and (\ref{eq.fxk-f*.geq1}), we can obtain
	\begin{equation*}
	\begin{aligned}
	&\|\nabla^T f(\bm{x}(t_k))\mathcal{L_O}\|^2\left\|\sum_{i=2}^{n}b_i(t_k)\bm{v}_i\right\|^2\\
	\geq& (f(\bm{x}(t_k))-f(\bm{x}^*))^2\\
	\geq& \frac{l_0}{2}\lambda_2(\mathcal{L_O}^T\mathcal{L_O})\left\|\sum_{i=2}^{n}b_i(t_k)\bm{v}_i\right\|^2(f(\bm{x}(t_k))-f(\bm{x}^*)),
	\end{aligned}
	\end{equation*}
	which gives
	\begin{equation}\label{eq.grad.g.norm}
	\begin{aligned}
	&\|\nabla^T f(\bm{x}(t_k))\mathcal{L_O}\|^2\\
	\geq& \frac{l_0}{2}\lambda_2(\mathcal{L_O}^T\mathcal{L_O})(f(\bm{x}(t_k))-f(\bm{x}^*)).
	\end{aligned}
	\end{equation}
	Substituting (\ref{eq.grad.g.norm}) into (\ref{eq.dV}) yields
	\begin{equation}\label{eq.dV2}
	\begin{aligned}
	&V(t_{k+1})-V(t_{k})\\
	\leq&-\frac{1}{4}\|\bm{e_\psi}(t_k)\|^2-\frac{\beta l_0\lambda_2(\mathcal{L_O}^T\mathcal{L_O})}{8}(f(\bm{x}(t_k))-f(\bm{x}^*))\\
	\leq&-\varepsilon V(t_k)\\
	\end{aligned}
	\end{equation}
	with
	\begin{equation*}
	\begin{aligned}
	\varepsilon=\min\{\frac{1}{4\|W\|},\frac{\beta l_0\lambda_2(\mathcal{L_O}^T\mathcal{L_O})}{8}\}.
	\end{aligned}
	\end{equation*}
	
	Then, one can obtain that $V(t_k)\leq(1-\varepsilon)^kV(0)$, and thus $
	\lim\limits_{k\rightarrow\infty}f(\bm{x}(t_k))=f(\bm{x}^*).
	$
	The above discussion illustrates the convergence properties
	only at sampling time instants
	$t_k, k {\in}\mathbb{N}$. Next, we focus on those at
	$t{\in}[t_k, t_{k+1})$. From (\ref{eq.law.x.i}), one can observe that $\bm{x}(t){=}\bm{x}(t_k),t{\in}[t_k,t_{k+1})$. Then, one can conclude that
	\begin{equation}\label{lim ft}
	\lim\limits_{t\in[t_k,t_{k{+}1}),k\rightarrow\infty}f(\bm{x}(t)){=}f(\bm{x}^*).
	\end{equation}
	From \dref{SamplingTimeIntervalSec2}, one can calculate that $\sum_{i=1}^{\infty}T_k=T_c$. The above analysis indicates that
	$f\left(\bm{x}(T_c)\right)=f(\bm{x}^*)$, which means that $\bm{x}$ will converge to the optimal value at the specific time $T_c$ under the proposed algorithm.	Recalling Lemma \ref{lemma.equality.constraint}, one can conclude that the proposed algorithm (\ref{eq.law.x.i}), (\ref{SamplingTimeSec2}) and (\ref{SamplingTimeIntervalSec2}) solves Problem \ref{problem}.
\end{proof}
\begin{remark}\rm\label{remark.advantages}
	Unlike the existing results on fixed-time distributed optimization \cite{ChenLiAuto2018ConDis,fixed-time-ningboda} that only the upper bound of the settling time can be conservatively estimated, the exact settling time can be precisely pre-specified in this paper {by using the sampling interval sequence (\ref{SamplingTimeIntervalSec2}), which  is inspired by the motion-planning-based specified-time consensus protocol in \cite{zhaoyu}.}  Moreover, different from \cite{ChenLiAuto2018ConDis,fixed-time-ningboda} where the upper bound of the settling time depends on algorithm parameters and communication topologies, the settling time in this paper can be determined independently, since it is totally decoupled with the algorithm parameters and communication topologies, which is a significant improvement for practical applications.
\end{remark}

\begin{remark}\rm\label{remark.sampling interval}
	From the proof of Theorem 1, one can observe that the design of the sequence of sampling time instants is not unique. Under the proposed framework, the only requirement  about the sampling interval sequence is that the corresponding series converges to $T_c$, i.e., 
	$\sum_{k=1}^{\infty}T_k{=}T_c$. For example,
	the sampling interval sequence can be also arranged as $(T_k)_{k=0}^\infty{:=}T_c(1-b)b^{k-1}$ with $b{\in}(0,1)$ being a constant, which  corresponds to the well-known power series.
	Note that $\lim_{k\rightarrow\infty}T_k=0$ is a necessary condition for the series 	$\sum_{k=1}^{\infty}T_k$ to be convergent.
	The sampling interval sequence whose series is convergent will induce Zeno behavior, since  infinite
	times of sampling  happen in a finite time interval, and the sampling interval tends to zero at the ending phase. 
However, it should be clarified that even in such a case, a significant improvement is still achieved from the perspective of communications, compared with the existing fixed-time distributed optimization algorithms in \cite{fixed-time-ningboda,ChenLiAuto2018ConDis} which require continuous-time neighboring information during the whole process.
\end{remark}

Next, we will show how to modify the sampling interval sequence to avoid Zeno behavior, such that the distributed optimization algorithm can be implemented into practical discrete-time communicating networked systems. One can modify the sampling interval sequence as follows:
\begin{equation}\label{SamplingTimeIntervalPractical}
T_k=\begin{cases}
\frac{6}{(\pi k)^2}T_{c}, \quad& k=1,\cdots,k_\epsilon,\\
\epsilon,& k=k_\epsilon+1,k_\epsilon+2,\cdots,
\end{cases}
\end{equation}
where $\epsilon$ is a small positive constant, and $k_\epsilon$ is a sufficiently large positive integer, which will affect the convergence accuracy. In \dref{SamplingTimeIntervalPractical}, periodical sampling is taken to replace the previous sampling interval after $k_\epsilon$ times of sampling. Therefore, $T_k$ defined in \dref{SamplingTimeIntervalPractical} will no longer converge to zero and thus the Zeno behavior is avoided.

\begin{corollary}\label{TheoremPartialContain}
	Suppose that Assumption \ref{AssumptionLip} holds, and graph $\mathcal{G}$ is strongly connected. Then, the distributed algorithm \dref{eq.law.x.i},  \dref{SamplingTimeSec2} and \dref{SamplingTimeIntervalPractical} approximately solves Problem \ref{problem} with the convergence accuracy
	\begin{equation}\label{eq.fTc-f*}
	\begin{aligned}
&f(\bm{x}(T_c))-f(\bm{x}^*)\\
\leq&(1-\varepsilon)^{k_\epsilon}\big(\bm{e_\psi}(0)^TW\bm{e_\psi}(0)+f(\bm{x}(0))-f(\bm{x}^*)\big)
	\end{aligned}
	\end{equation} if $\beta$ satisfies inequality (\ref{eq.beta}).
\end{corollary}

\begin{remark}\rm\label{remark.practical}
	To avoid Zeno behavior, a modified sampling time interval is proposed in (\ref{SamplingTimeIntervalPractical}), which makes the proposed algorithm able to be implemented into practical discrete-time communicating networked systems.
	Inequality (\ref{eq.fTc-f*}) reveals that the convergence accuracy depends on the total steps of sampling within the specified time interval. 
	Although
	the algorithm will no longer accurately converge to the optimal point at the specified time, one can always properly arrange the sampling interval sequence according to (\ref{SamplingTimeIntervalPractical}) with the total steps of sampling big enough, such that the convergence precision of the proposed algorithm can meet the practical demands.
\end{remark}
\section{Specified-time distributed optimization for undirected MANs}\label{SectionFour}
In Section \ref{SectionThree}, a specified-time distributed optimization algorithm has been proposed for general directed MANs. A design structure is established to meet the global equality constraint, under which  each agent updates its  auxiliary variable based on the estimated derivations of the local objective functions of its out-neighbors to make its own state converge to the optimal value.

In the special case of undirected MANs, the Laplacian matrix is symmetric, and the out-neighbors of an agent are also the in-neighbors of the agent.
Under this circumstance, it is natural to use the neighbouring information  directly rather than the estimated one, such that the observer design can be removed and the computational cost can be greatly reduced.
In view of this, we design the following reduced-order continuous-time algorithm based on sampled-data communication:
\begin{equation}\label{eq.law.x.i.2}
\begin{aligned}
x_{i}(t)&=x_{i}(0){-}\sum_{j=1}^n a_{ij}(\xi_{i}(t_k){-}\xi_{j}(t_k)),\\
\dot{\xi}_{i}(t)&=\frac{\beta}{t_{k{+}1}{-}t_{k}}\sum_{j=1}^n a_{ij}\left(\frac{\mathrm{d}f_{i}}{\mathrm{d}x_{i}}(x_{i}(t_k))-\frac{\mathrm{d}f_{j}}{\mathrm{d}x_{j}}(x_{j}(t_k))\right),\\
t\in&[t_k,t_{k+1}),~ k\in\mathbb{N},
\end{aligned}
\end{equation}
where $\beta$ is a positive constant satisfying $\beta{\leq} 1/(l\|\mathcal{L\|}^2)$,  $\xi_i(0){=}0$, and other variables are defined the same as in the previous section.


\begin{theorem}
	Suppose that Assumptions \ref{AssumptionLip} holds, and graph $\mathcal{G}$ is undirected and connected. Then, the distributed algorithm (\ref{SamplingTimeSec2}), (\ref{SamplingTimeIntervalSec2}) and (\ref{eq.law.x.i.2})
	solves Problem \ref{problem}.
\end{theorem}
\begin{proof}
From (\ref{eq.law.x.i.2}), one has
	\begin{equation}\label{eq.x.compact.undirected}
	\bm{x}(t){=}\bm{x}(0){-}\mathcal{L}\bm{\xi}(t_k),\quad t{\in}[t_k,t_{k+1}),\;k{\in}\mathbb{N},
	\end{equation}
which indicates that $\mathbf{1}^T\bm{x}(t){=}\mathbf{1}^T\bm{x}(0){=}C,$ $\forall t{\geq}0$.
	Integrating the second equality in \dref{eq.law.x.i.2} from $t_k$ to $t_{k+1}$ yields
	\begin{equation*}
	\xi_i(t_{k+1}){-}\xi_i(t_k)=\beta\sum_{j\in N_{i}}a_{ij}\left(\frac{\mathrm{d}f_{i}}{\mathrm{d}x_{i}}(x_{i}(t_k))-\frac{\mathrm{d}f_{j}}{\mathrm{d}x_{j}}(x_{j}(t_k))\right),
	\end{equation*}
which can be collectivity rewritten as
	\begin{equation*}\label{xi_k+1 g}
	\bm{\xi}(t_{k+1})-\bm{\xi}(t_{k})=\beta \mathcal{L}\nabla f(\bm{x}(t_k)).
	\end{equation*}
Then, one can obtain the following inequality similar to (\ref{eq.df1}):
	\begin{equation}\label{eq.ineq1}
	\begin{aligned}
&	f(\bm{x}(t_{k+1}))-f(\bm{x}(t_k))\\
	\leq & \nabla^Tf(\bm{x}(t_k))(\bm{x}(t_{k+1})-\bm{x}(t_k))\\
	&+\frac{{l}}{2}(\bm{x}(t_{k+1})-\bm{x}(t_k))^T(\bm{x}(t_{k+1})-\bm{x}(t_k))\\
	\leq & -\beta\|{\mathcal{L}}\nabla f(\bm{x}(t_k))\|^2
+\frac{{l\beta^2}}{2}\|\mathcal{L}\|^2\|{\mathcal{L}}\nabla f(\bm{x}(t_k))\|^2\\
	{\leq}&-\frac{\beta}{2}\|{\mathcal{L}}\nabla f(\bm{x}(t_k))\|^2.
	\end{aligned}
	\end{equation}
Similar to (\ref{eq.grad.g.norm}), one has
\begin{equation}\label{eq.grad.g.norm.2}
\begin{aligned}
&\|{\mathcal{L}}\nabla f(\bm{x}(t_k))\|^2\geq \frac{l_0\lambda_2^2(\mathcal{L})}{2}(f(\bm{x}(t_k))-f(\bm{x}^*)),
\end{aligned}
\end{equation}
where $\lambda_{2}(\mathcal{L})$ denotes the minimal non-zero eigenvalue of $\mathcal{L}$.
Combining (\ref{eq.ineq1}) and (\ref{eq.grad.g.norm.2}), one has
	\begin{equation*}\label{eq.dV2.sec4}
\begin{aligned}
&	f(\bm{x}(t_{k+1}))-f(\bm{x}(t_k))\\
{\leq}&-\frac{\beta l_0\lambda_2^2(\mathcal{L})}{4}(f(\bm{x}(t_k))-f(\bm{x}^*)).
\end{aligned}
\end{equation*} Then, it is clear that
\begin{equation*}
\begin{aligned}
&f(\bm{x}(t_k))-f(\bm{x}^*)\\
\leq&\left(1-\frac{\beta l_0\lambda_2^2(\mathcal{L})}{4}\right)^k\left(f(\bm{x}(0))-f(\bm{x}^*)\right),
\end{aligned}
\end{equation*}
 from which one can derive that $
\lim\limits_{k\rightarrow\infty}f(\bm{x}(t_k)){=}f(\bm{x}^*).
$
Then, taking similar steps as  in the proof of Theorem 1, one has
$f(\bm{x}(T_c)){=}f(\bm{x}^*)$.	Recalling (\ref{eq.x.compact.undirected}), one can conclude that  Problem \ref{problem} is solved.
\end{proof}

\begin{corollary}\label{coro.undirected}
	Suppose that Assumptions \ref{AssumptionLip} holds, and graph $\mathcal{G}$ is undirected and connected. Then, the distributed algorithm \dref{SamplingTimeSec2}, \dref{SamplingTimeIntervalPractical} and \dref{eq.law.x.i.2} approximately solves Problem \ref{problem} with the convergence accuracy  \begin{equation*}
	\begin{aligned}
	&f(\bm{x}(T_c))-f(\bm{x}^*)\\
	\leq&\left(1-\frac{\beta l_0\lambda_2(\mathcal{L}^2)}{4}\right)^{k_\epsilon}\left(f(\bm{x}(0))-f(\bm{x}^*)\right).
	\end{aligned}
	\end{equation*}
\end{corollary}

\begin{figure}[h]
	\centering
	\includegraphics[width=1.2in]{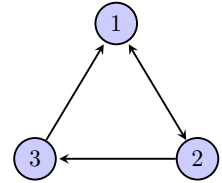}
	\caption{The communication topology of the generators.}
	\label{fig.topology.di}
\end{figure}

\section{Numerical simulations}\label{SectionFive}
In this section, a numerical example is provided to verify our established theoretical results.
Consider a power system consisting of three generators with the communication topology shown in Fig. \ref{fig.topology.di}, which is a general unbalanced directed graph. Denote the active power generated by the three generators as $x_1,x_2,x_3$, respectively.
From \cite{ChenLiAuto2018ConDis}, the cost function of each  generator
can be generally approximated by: $f_1(x_1)=0.096 x_1^2+ 1.22 x_1 +51,$ $f_2(x_2)= 0.072 x_2^2+3.41x_2 +31$, and $f_3(x_3)=0.105 x_3^2+2.53x_3 + 78$.
The  total load demand is
$420$MW, and the initial values $x_1(0){=}x_2(0){=}x_3(0){=}140$MW.
The economic dispatch problem of the power system is to determine the active power supplied by each generator to  minimize the total cost while meeting the total load demand, which can be described by 
\begin{equation}\label{simu.problem}
\begin{aligned}
&\mathrm{Minimize}\;\sum_{i=1}^{3}f_{i}(x_{i}),\\
& \mathrm{subject}~~\mathrm{to}~ \sum_{i=1}^{3}x_{i}(t)=420\text{MW}.
\end{aligned}
\end{equation}

\begin{figure}[h]
	\centering
	\includegraphics[width=3.5in]{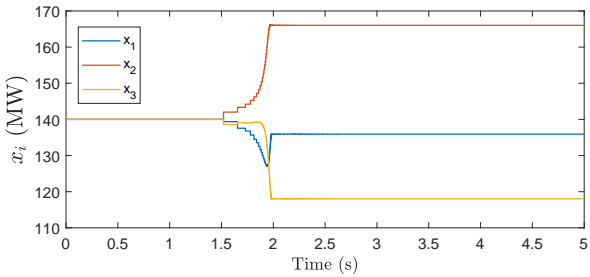}
	\caption{Evolution of $x_1,x_2,x_3$ during 0s--5s.}
	\label{fig.x}
\end{figure}

\begin{figure}[h]
	\centering
	\includegraphics[width=3.5in]{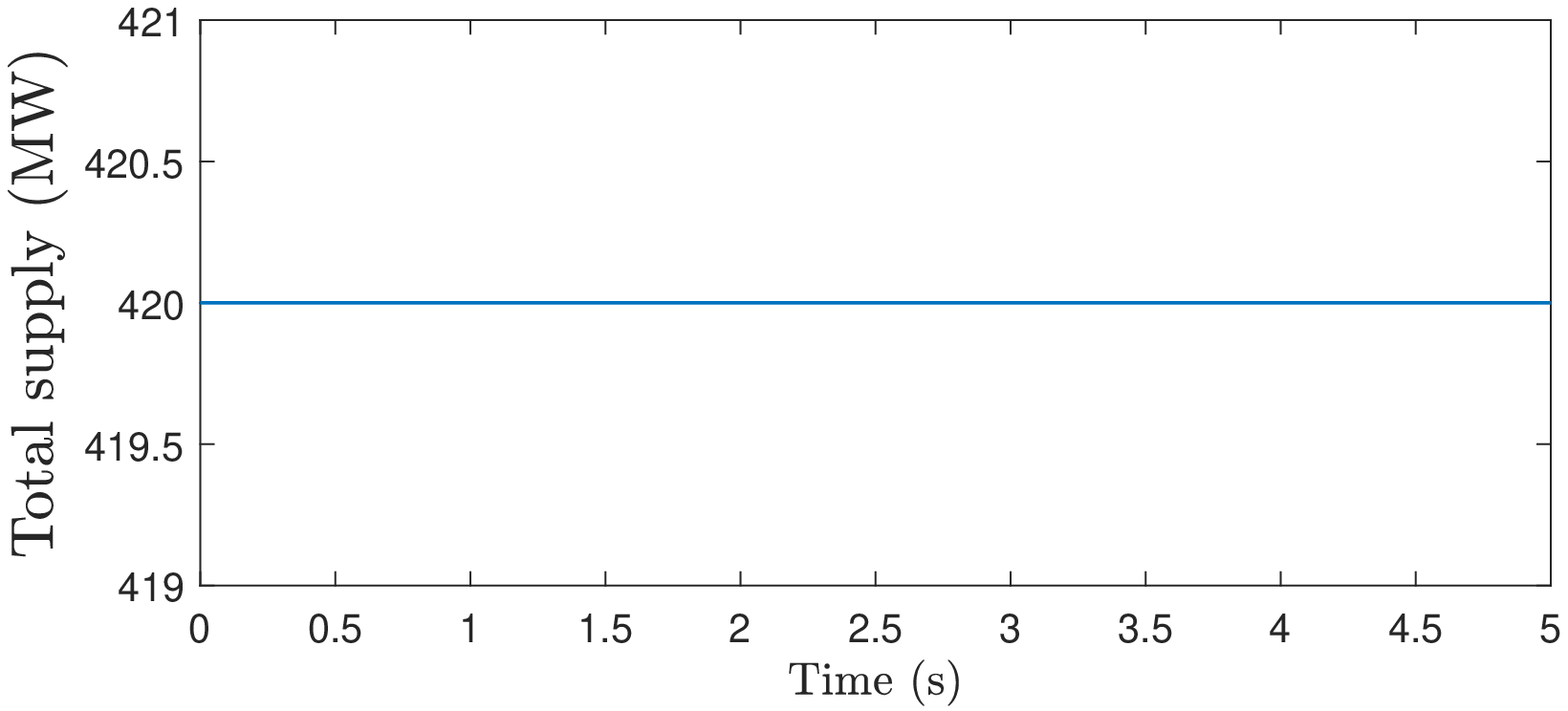}
	\caption{Evolution of the total supply $\sum_{i=1}^3x_i$ during 0s--5s.}
	\label{fig.sum_x}
\end{figure}

\begin{figure}[h]
	\centering
	\includegraphics[width=3.5in]{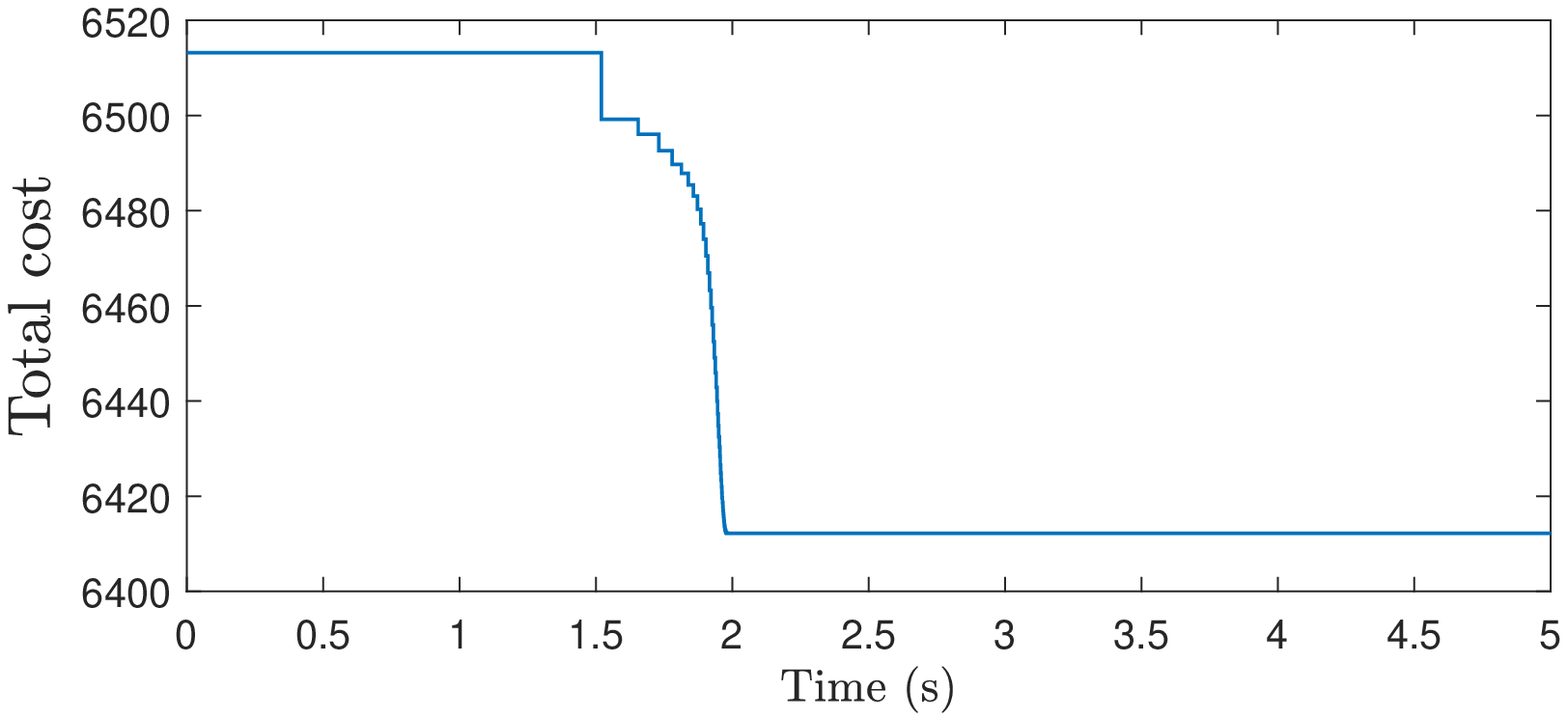}
	\caption{Evolution of the total cost $\sum_{i=1}^3f_i$ during 0s--5s.}
	\label{fig.totalcost}
\end{figure}

\begin{figure}[h]
	\centering
	\includegraphics[width=3.5in]{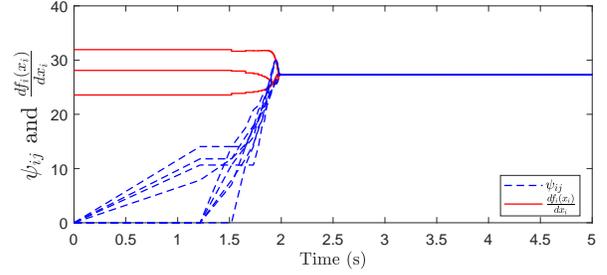}
	\caption{Evolution of $\psi_{ij}$ and $\frac{df_i(x_i)}{dx_i}$ $(i,j{=}1,2,3)$ during 0s--5s.}
	\label{fig.psi}
\end{figure}

To find the solution to problem  \dref{simu.problem} at the pre-specified time $T_c=2$s, we use the proposed algorithm \dref{eq.law.x.i}, which allows the agents exchange information only at some discrete-time instants. Arranged the sampling time sequence as \dref{SamplingTimeSec2} and \dref{SamplingTimeIntervalPractical} with $k_\epsilon=80$, and $\epsilon=0.01$. The results are shown in Figs.\ref{fig.x}-\ref{fig.psi}. The values of $x_1,x_2,x_3$ and the total cost at $5s$ are
$135.9293$, $166.0307$, $118.0401$ and $6412.187283$ respectively, which are  the same as the optimal solution found in \cite{ChenLiAuto2018ConDis}. Figs. \ref{fig.x} and \ref{fig.totalcost} show that $x_1,x_2,x_3$  and the total cost almost converge to the optimal values at the pre-specified time $2$s respectively.
To be more specific, the value of the total cost at $2$s is $6412.187397$, showing the high accuracy of the proposed specified-time distributed optimization algorithm with the modified Zeno-free sampling time sequence. Fig. \ref{fig.sum_x} illustrates that the equality constraint in problem \dref{simu.problem} is satisfied during the whole course. One can observe from Fig. \ref{fig.psi} that $\psi_{ij},\frac{df_i}{dx_i}(x_i),i,j{=}1,2,3$ reach consensus at $2$s, which is in accordance with the theoretical analysis that the components of the gradient of the global objective function are equal to each other at the optimal point and the gradient estimation by each agent will converge to the actual gradient at the specified time.

\section{Conclusion}\label{SectionSix}
In this paper, specified-time distributed optimization problem subject to the global equality constraint is investigated. A new distributed optimization algorithm is first proposed for general directed MANs, which can exactly reach the optimal solution at specified time under sampled-data communications. Under this algorithm, the settling time does not rely on algorithm parameters or communication topologies, which can be pre-assigned independently. For the special case of undirected MANs, a reduced-order distributed optimization algorithm is also designed, such that the computational cost would be greatly saved. Future work can be done to design specified-time distributed optimization algorithms for MANs with 
more practical constraints on decision variables.




\begin{thebibliography}{99}
	\bibitem{survey2018}
	A. Nedi\'c and J. Liu,
	``Distributed optimization for control,''
	\emph{Annual Review of Control, Robotics, and Autonomous Systems}, vol. 1, pp. 77--103, 2018.
	
	
	\bibitem{survey-yangtao}
	T. Yang, X. Yi, J. Wu, Y. Yuan, D. Wu, Z. Meng, Y. Hong, H. Wang, Z. Lin, and K. H. Johansson,
	``A survey of distributed optimization,''
	\emph{Annual Reviews in Control},
	vol. 47, pp. 278--305, 2019.
	
	
	
	\bibitem{subgradient2009}
	A. Nedi\'c and A. Ozdaglar,
	``Distributed subgradient methods for multi-agent optimization,''
	\emph{ IEEE Transactions on Automatic Control}, vol. 54, no. 1, pp. 48--61, 2009.
	
	
	
	
	\bibitem{Primal-dual2016}
	J. Lei, H. Chen, and H. Fang,
	``Primal-dual algorithm for distributed constrained optimization,'' \emph{Systems and Control Letters}, vol. 96, pp. 110--117, 2016.
	
	
	
	
	\bibitem{push-sum-yangtao}
	T. Yang, J. Lu, D. Wu, J. Wu, G. Shi, Z. Meng, and K. H. Johansson,
	``A distributed algorithm for economic dispatch over time-varying directed networks with delays,''
	\emph{IEEE Transactions on Industrial Electronics}, vol. 64, no. 6, pp. 5095--5106, 2017.
	
	
	\bibitem{push-pull-pushi}
	S. Pu, W. Shi, J. Xu and A. Nedic,
	``Push-pull gradient methods for distributed optimization in networks,''
	\emph{IEEE Transactions on Automatic Control}, doi: 10.1109/TAC.2020.2972824.
	
	\bibitem{gfh-cluster-consensus}
	F. Guo, C. Wen, J. Mao, G. Li and Y. Song, ``A distributed hierarchical algorithm for multi-cluster constrained optimization,''
	\emph{Automatica}, vol. 77, pp. 230--238, 2017.
	
	\bibitem{gfh-cluster-dispatch}
	F. Guo, C. Wen, J. Mao, J. Chen and Y. Song, ``Hierarchical decentralized optimization architecture for economic dispatch: a new approach for large-scale power system,''
	\emph{IEEE Transactions on Industrial Informatics}, vol. 14, no. 2, pp. 523--534, 2018.
	
	\bibitem{continous-pi}
	B. Gharesifard and J. Cort\'es,
	``Distributed continuous-time convex optimization on weight-balanced digraphs,''
	\emph{IEEE Transactions on Automatic Control},  vol. 59, no. 3, pp. 781--786, 2014.
	
	\bibitem{zero-gradient-sum2012}
	J. Lu and C. Y. Tang, ``Zero-gradient-sum algorithms for distributed convex optimization: the continuous-time case,''
	\emph{IEEE Transactions on Automatic Control}, vol. 57, no. 9, pp. 2348--2354, 2012.
	
	\bibitem{zero-gradient-sum2016renwei}
	W. Chen and  W. Ren,
	``Event-triggered zero-gradient-sum distributed consensus
	optimization over directed networks,''
	\emph{Automatica}, vol. 65, pp. 90--97, 2016.
	
	
	
	\bibitem{discrete-time-finite-time1}
	T. Yang, D. Wu, Y. Sun,  and J. Lian,
	``Minimum-time consensus based approach for power system applications,''\emph{IEEE Transactions on Industrial Electronics}, vol. 63, no. 2, pp. 1318--1328, 2016.
	
	
	\bibitem{discrete-time-finite-time2}
	V. S. Mai and E. H. Abed, ``Local prediction for enhanced convergence of distributed optimization algorithms,'' \emph{IEEE Transactions on Control of Network Systems}, vol. 5, pp. 1962--1975, 2018.
	
	
	\bibitem{finite-time}
	P. Lin, W. Ren, and J. A. Farrell,
	``Distributed continuous-time optimization:
	nonuniform gradient gains, finite-time convergence, and convex constraint set,''
	\emph{IEEE Transactions on Automatic Control}, vol. 62, no. 5, pp. 2239--2253, 2017.
	
	\bibitem{ChenLiAuto2018ConDis}
	G. Chen and Z. Li,
	``A fixed-time convergent algorithm for distributed convex optimization in multi-agent systems,''
	\emph{Automatica},
	vol. 95,  pp. 539--543, 2018.
	
	
	\bibitem{fixed-time-ningboda}
	B. Ning, Q. Han and Z. Zuo,
	``Distributed optimization for multiagent systems: an edge-based fixed-time consensus approach,''
	\emph{IEEE Transactions on Cybernetics}, vol. 49, no. 1, pp. 122--132, 2019.
	
	
	\bibitem{renwei2007Information-consensus}
	W. Ren, R. W. Beard and E. M. Atkins, ``Information consensus in multivehicle cooperative control,''
	\emph{IEEE Control Systems Magazine}, vol. 27, no. 2, pp. 71--82, 2007.
	
	
	

\bibitem{zhaoyu}
Y.~{Zhao}, Y.~{Liu}, G.~{Wen}, W.~{Ren}, and G.~{Chen}, ``Designing distributed
specified-time consensus protocols for linear multiagent systems over
directed graphs,'' \emph{IEEE Trans. Autom. Control}, vol.~64,
no.~7, pp. 2945--2952, 2019.	
	
	
	%
	
	%
	%
	
\end{thebibliography}
\end{document}